\documentclass{amsart}
\usepackage{amsmath,amsthm}
\usepackage{amsfonts,amssymb}

\usepackage{enumerate}
\usepackage{multirow}

\newtheorem{thm}{Theorem}[section]
\newtheorem{cor}[thm]{Corollary}
\newtheorem{lem}[thm]{Lemma}
\newtheorem{prop}[thm]{Proposition}

\theoremstyle{remark}

\def\f{\frac}

 \def\a{{\alpha}} 
 \def\b{{\beta}}
 \def\g{{\gamma}}
 
 \def\t{{\theta}}
 \def\l{{\lambda}}

 \def\la{{\langle}}
 \def\ra{{\rangle}}

 \def\CV{{\mathcal V}}
 \def\CW{{\mathcal W}}

 \def\NN{{\mathbb N}}

 \def\RR{{\mathbb R}}

\newif\ifpdf
\ifx\pdfoutput\undefined
  \pdffalse
\else
  \pdftrue
\fi

\ifpdf
  \usepackage[pdftex]{graphicx}
  \DeclareGraphicsExtensions{.pdf,.jpg,.png}
\else
  \usepackage{graphicx}
\fi

\begin{document}
\title[Minimal cubature rules on unbounded domain]
{Minimal cubature rules on an unbounded domain}
 
\author{Yuan Xu}
\address{Department of Mathematics\\ University of Oregon\\
    Eugene, Oregon 97403-1222.}\email{yuan@math.uoregon.edu}
\thanks{The work of the third author was supported in part by 
NSF Grant DMS-1106113.}

\date{\today}

\begin{abstract}
A family of minimal cubature rules is established on an unbounded domain, which 
is the first such family known on unbounded domains. The nodes of such cubature
rules are common zeros of certain orthogonal polynomials on the unbounded domain,
which are also constructed. 
\end{abstract}

\maketitle

\section{Introduction}
\setcounter{equation}{0}

In two or more variables, few families of minimal cubature rules are known in the literature,
none on unbounded domains. The purpose of this note is to record a family of minimal cubature 
rules on an unbounded domain. 

The precision of a cubature rule is usually measured by the degrees of polynomials that can be
evaluated exactly. For a nonnegative integer $m$, we denote be $\Pi_m^2$ the space of polynomials 
of degree at most $m$. Let $\Omega$ be a domain in $\RR^2$ and let $W$ be a non-negative weight 
function on $\Omega$. A cubature rule of precision $2n-1$ with respect to $W$ is a finite sum that 
satisfies
\begin{equation} \label{cuba-generic}
        \int_\Omega f(x,y) W(x,y) dxdy = \sum_{k=1}^N \l_k f(x_k,y_k), 
             \qquad \forall f\in \Pi_{2n-1}^2, 
\end{equation}
and there exists at least one function $f^*$ in $\Pi_{2n}^2$ such that the equation \eqref{cuba-generic} 
does not hold. A cubature rule in the form of \eqref{cuba-generic} is called {\it minimal}, if its number of
nodes is the smallest among all cubature rules of the same precision for the same integral. 

It is well known that the number of nodes, $N$, of \eqref{cuba-generic} satisfies (cf. \cite{My, St}), 
\begin{equation}\label{lbdGaussian}
   N \ge \dim \Pi_{n-1}^2 =  \frac{n (n+1)}{2}.
\end{equation}
A cubature rule of degree $2n-1$ with $N =  \dim \Pi_{n-1}^2 $ is called Gaussian. In contrast to 
the Gaussian quadrature of one variable, Gaussian cubature rules rarely exists and there are two 
family of examples known \cite{LSX, SX}, both on bounded domains. It is known that they do not 
exist if $W$ is centrally symmetric, which means that $W(x) = W(-x)$ and $- x \in \Omega$ 
whenever $x \in \Omega$. In fact, in the centrally symmetric case, the number of nodes of 
\eqref{cuba-generic} satisfies a better lower bound \cite{M}, 
\begin{equation}\label{lwbd}
     N \ge \dim \Pi_{n-1}^2 + \left \lfloor \frac{n}{2} \right \rfloor 
                   = \frac{n(n+1)}{2} +  \left \lfloor \frac{n}{2} \right \rfloor. 
\end{equation}

A cubature rule whose number of nodes attains a known lower bound is necessarily minimal. It turns 
out that a family of weight functions $\CW_{\a,\b, \pm \f12}$ defined by 
\begin{equation} \label{CWab}
 \CW_{\a,\b,\pm \frac12}(x,y) : = |x+y|^{2 \a+1} |x- y|^{2 \beta +1}
          (1-x^2)^{\pm \frac12}(1-y^2)^{\pm \frac12}, \quad \a,\b > -1,
\end{equation}
on $[-1,1]^2$ admits minimal cubature rules of degree $4n-1$ on the square $[-1,1]^2$, which was 
established in \cite{MP} for $\a =\b = 0$ (see also \cite{BP, X94}) and in \cite{X12a} for $(\a,\b) \ne (0,0)$.
There are not many other cases for which minimal cubature rules are known to exist for all $n$ and 
none in the literature that are known for unbounded domains. 

To establish the minimal cubature rules for $\CW_{\a,\b,\pm 12}$, the starting point in \cite{X12a} is 
the Gaussian cubature rules in \cite{SX} and it amounts to a series of changing of variables from the 
product Jacobi weight function on the square $[-1,1]^2$ to the weight function $\CW_{\a,\b, \pm 12}$. 
The procedure works for general product weight function on the square. Moreover, as we shall shown
in this note, that the procedure also works for an unbounded domain, which leads to our main results
in this note. 
The minimal cubature rules are known to be closely related to orthogonal polynomials, as 
their nodes are necessarily zeros of certain orthogonal polynomials. We will discuss this connection and 
construct an explicit orthogonal basis on our unbounded domain in Section 3. 

\section{Gaussian Cubature rules on an unbounded domain}
\setcounter{equation}{0}

Let $W$ be a nonnegative weight function on a domain $\Omega \subset \RR^2$. A polynomial 
$P \in \Pi_n^d$ is an orthogonal polynomial of degree $n$ with respect to $W$ if 
$$
   \int_{\Omega} P_{k,n}(x,y) q(x,y) W(x,y) dxdy = 0, \qquad \forall q \in \Pi_{n-1}^2.
$$
Let $\CV_n^2$ be the space of orthogonal polynomials of degree exactly $n$. Then $\dim \CV_n^2 
= n +1$. The Gaussian cubature rules can be characterized in terms of the common zeros of 
elements in $\CV_n^2$ (\cite{My,St}).  

\begin{thm}
Let $\{P_{k,n}: 0 \le k \le n\}$ be a basis of $\CV_n^2$. Then a Gaussian cubature rule of degree 
$2n-1$ for the integral against $W$ exists if and only if its nodes are the common zeros of 
$P_{k,n}$, $0 \le k \le n$.
\end{thm}

We now describe a family of Gaussian cubature rules on an unbounded domain. Let $w(x)$ be a 
nonnegative weight function defined on the unbounded domain $[1, \infty)$ and let $c_w$ denote 
its normalization constant defined by $c_w \int_{1}^\infty w(x) dx =1$. Let $p_n(w;x)$ be the orthogonal 
polynomial of degree $n$ with respect to $w$ and let $x_{1,n}, x_{2,n},\ldots, x_{n,n}$ be the zeros of 
$p_n(w;x)$. It is well known that $x_{k,n}$ are real and distinct points in $[1,\infty)$. 
The Gaussian quadrature rule for the integral against $w$ is given by 
\begin{equation} \label{Gauss-quad}
  c_w \int_{1}^\infty f(x) w(x) dx = \sum_{k=1}^n \l_{k,n} f(x_{k,n}), \quad f \in \Pi_{2n-1},
\end{equation}
where $\Pi_{2n-1}$ denotes the space of polynomials of degree $2n-1$ in one variable and the 
weights $\l_{k,n}$ are known to be all positive. 

A typical example of $w$ is the shifted Laguerre weight function 
$$
 w_\a(x) := (x-1)^\a e^{-x+1}, \quad \a > -1,
$$
for which the orthogonal polynomial $p_n(w_\a;x)$ is the Laguerre polynomial $L_n^\a(x-1)$ with
argument $x-1$. 

The unbounded domain on which our Gaussian cubature rules live is given by 
\begin{equation}   \label{Omega} 
   \Omega: = \{(u,v): 0 < u-1 < v < u^2/4\}, 
\end{equation}
which is bounded by a parabola and a line for $u \ge 1$ and $v \ge 2$, which is the shaded area
depicted in Figure 1. 
\begin{figure}[ht]
\includegraphics[scale=0.48]{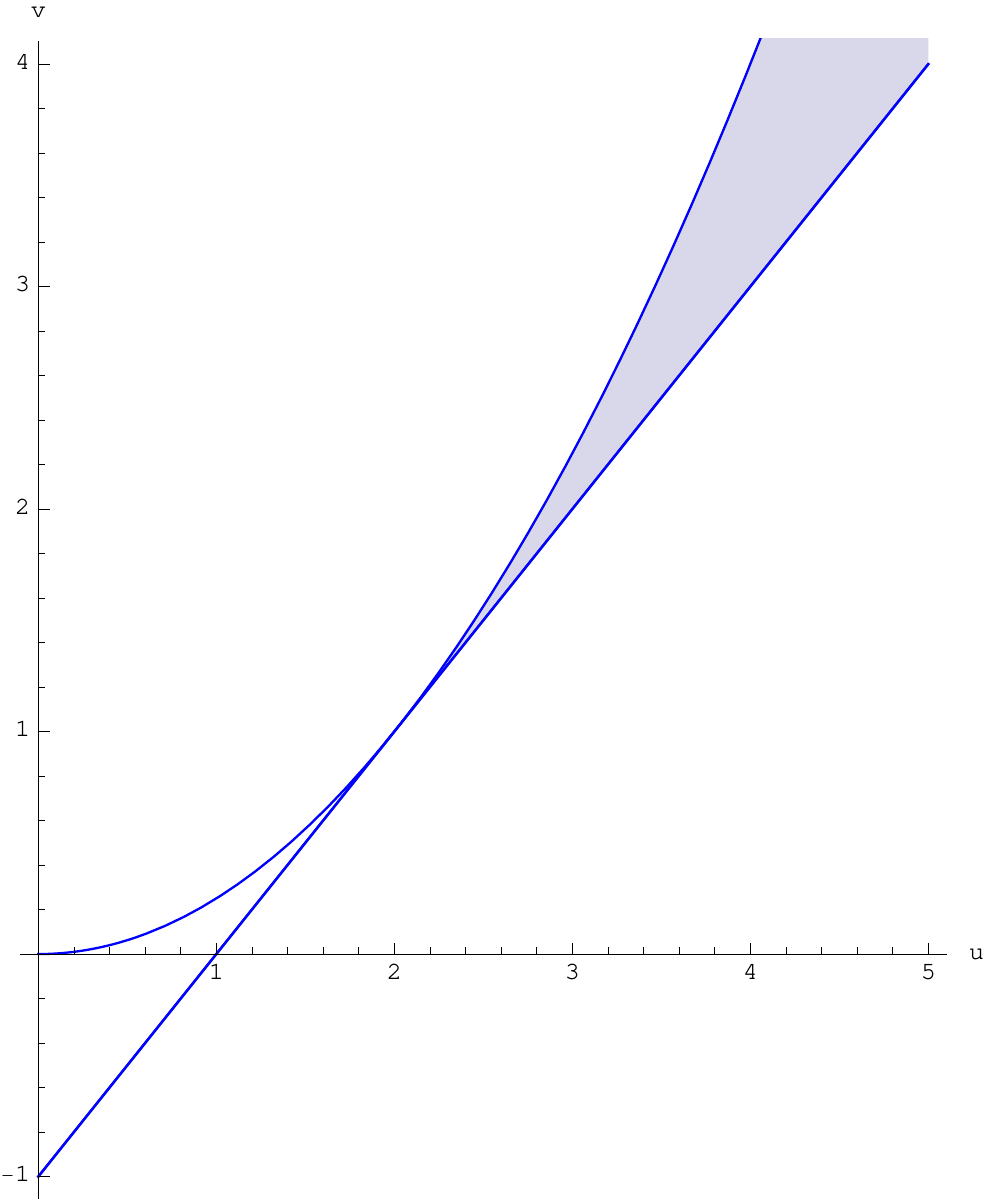}
\caption{Domain $\Omega$}
\label{figure:region} 
\end{figure}

The function $w(x)w(y)$ is evidently a symmetric function in $x$ and $y$. For $\g > -1$, we define the 
weight
\begin{equation}   \label{Wgamma} 
    W_\gamma(u,v) :=   w(x)w(y) |u^2 - 4 v|^{\gamma}, \quad \g > -1, 
\end{equation}
where the variables $(x,y)$ and $(u,v)$ are related by 
\begin{equation}   \label{u-v} 
        u = x+y, \quad v= xy.
\end{equation}
The function $w(x)w(y)$ is obviously symmetric in $x$ and $y$, so that it can be written as a function 
of $(u,v)$ for $(x,y)$ in the domain 
$$
 \triangle: = \{(x,y): 1 < x < y < \infty\}.
$$
The function $W_\gamma(u,v)$ is the image of $w(x)w(y)|x-y|^{2\g+1}$ under the changing of 
variables $u =x+y$ and $v = xy$, which has a Jacobian $|x-y|$ and $|x-y| = \sqrt{u^2 - 4 v}$. When
$w = w_\a$ is the shifted Laguerre weight, we denote $W_\g$ by $W_{\a,\g}$, which is given 
explicitly by
\begin{equation*}
W_{\a,\g}(x,y) := (v - u +1)^\a e^{-u -2} (u^2 - 4 v)^\g. 
\end{equation*}

The changing of variables \eqref{u-v} immediately leads to the relation
\begin{align} \label{Para-Square}
  \int_{\Omega} f(u,v) W_\g(u,v) dudv & = 
         \int_{\triangle}  f(x+y, x y) w(x)w(y)|x-y|^{2\g+1} dxdy \\
       & = \f12 \int_{[1,\infty)^2}  f(x+y, x y) w(x)w(y)|x-y|^{2\g+1} dxdy, \notag
\end{align}
where the second equation follows from symmetry. Now, recall that $x_{k,n}$ denote the zeros of $p_n(w;x)$. 
We define
$$
    u_{k,j} = x_{k,n}+x_{j,n}, \qquad v_{k,j} = x_{k,n} x_{j,n}, \qquad 0 \le j \le k \le n. 
$$

\begin{thm} \label{thm:Gaussian}
For $W_{- \frac12}$ on $\Omega$, the Gaussian cubature rule of degree $2n-1$ is 
\begin{equation} \label{GaussCuba-}
   c_w^2 \int_{\Omega} f(u, v) W_{-\frac12}(u,v) du dv = \sum_{k=1}^n \mathop{ {\sum}' }_{j=1}^k  
          \l_k \l_j f(u_{j,k}, v_{j,k}), \quad f \in \Pi_{2n-1}^2, 
\end{equation}
where ${\sum}'$ means that the term for $j = k$ is divided by 2. For $W_{\frac12}$ on 
$\Omega$, the Gaussian cubature rule of degree $2n-3$ is 
\begin{equation} \label{GaussCuba+}
   c_w^2 \int_{\Omega} f(u, v) W_{\frac12}(u,v) du dv =  \sum_{k=2}^n \sum_{j=1}^{k-1}  
          \l_{j,k} f(u_{j,k}, v_{j,k}), \quad f \in \Pi_{2n-3}^2, 
\end{equation} 
where $\lambda_{j,k} = \l_j \l_k (x_{j,n} - x_{k,n})^2$. 
\end{thm}
 
The proof follows almost verbatim from the proof of Theorem 3.1 in \cite{X12a}. In fact, by 
\eqref{Para-Square}, the cubature rule \eqref{GaussCuba-} is equivalent to the cubature
rule for $w(x) w(y) dxdy$ on $[1,\infty)^2$ for polynomials $f(x+y,xy)$ with $f \in \Pi_{2n-1}^2$,
which are symmetric polynomials in $x$ and $y$. By the Sobolev's theorem on invariant cubature
rules \cite{Sobolev}, this cubature rule is equivalent to the product Gaussian cubature rules for 
$w(x)w(y)$ on $[1,\infty)^2$, which is the product of \eqref{Gauss-quad}. Thus, the cubature rule \eqref{GaussCuba-} follows from the product of Gaussian cubature rules under \eqref{u-v}. 

The existence of these cubature rules also follow from counting common zeros of the orthogonal 
polynomials with respect to $W_{\pm \f12}$. Indeed, a mutually orthogonal basis with respect to 
$W_{-\frac12}$ on $\Omega$ is given by 
\begin{equation} \label{OP-1/2}
    P_{k,n}^{(-\frac12)} (u,v) = p_n(x) p_k(y) + p_n(y) p_k(x), \qquad 0 \le k \le n, 
\end{equation}
and a mutually orthogonal basis with respect to $W_{\frac12}$ is given by 
\begin{equation} \label{OP+1/2}
    P_{k,n}^{(\frac12)} (u,v) = \frac{p_{n+1}(x) p_k(y) - p_{n+1} (y) p_k(x)}{x-y}, \quad 0 \le k \le n, 
\end{equation}
both families are defined under the mapping \eqref{u-v}. This was established in \cite{K74} for the
domain $[-1,1]^2$, but the proof can be easily extended to our $\Omega$. It is easy to see that 
the elements of $\{(u_{j,k}, v_{j,k}): 0 \le j \le k \le n\}$ are common zeros of $P_{k,n}^{(-\f12)}$, 
$0 \le k \le n$, and the cardinality of this set is $\dim \Pi_{n-1}^2$, which implies, by Theorem 2.1,
that the Gaussian cubature rules for $W_{-\f12}$ exists. The proof for $W_{\f 12}$ works similarly. 

\section{Minimal cubature rules on an unbounded domain}
\setcounter{equation}{0}

We are looking for cubature rules of degree $2n-1$ that satisfy the lower bound \eqref{lwbd},
which are necessarily minimal cubature rules. Such cubature rules are characterized by common 
zeros of a subspace of $\CV_n^2$ (\cite{M}). 

\begin{thm}\label{thm:minimalCuba}
A cubature rule whose number of nodes attains the lower bound \eqref{lwbd} exists if, and only if, 
its nodes are common zeros of $\lfloor \frac{n+1}{2} \rfloor +1$ orthogonal polynomials of degree $n$. 
\end{thm}

Let $w$ be the weight
function defined on $[1,\infty)$ and let $W_\g$ be the corresponding weight function in \eqref{Wgamma}
defined on $\Omega$ in \eqref{Omega}. We define a family of new weight functions by
\begin{equation}\label{CWgamma}
     \CW_\gamma (x,y): = W_\g (2 x y, x^2+y^2 -1)|x^2-y^2|, \qquad (x,y) \in G,
\end{equation}
where the domain $G$ is centrally symmetric and defined by 
$$
  G :=  (-\infty, -1]^2 \cup [1,\infty)^2.
$$
Thus, the weight function $\CW_\gamma$ is centrally symmetric on $G$. 

That $\CW_\gamma$ is well-defined on $G$ is established in the next lemma. Recall that 
$\triangle = \{ (x,y):  1 < x < y < \infty\}$. 

\begin{lem} \label{Int-Para-cube}
The mapping $(x,y)  \mapsto (2 x y , x^2+y^2 -1)$ is a bijection from $\triangle$ onto 
$\Omega$. Furthermore,
\begin{align} \label{Int-P-Q}
 \int_{\Omega} f(u,v) W_{\g} (u,v) du dv  = \int_{G} f(2xy, x^2+y^2 -1) \CW_{\g}(x,y) dx dy.
\end{align}
\end{lem}

\begin{proof}
Recall that if $(x,y) \in \triangle$, then $(x+y,xy) \in \Omega$ and the mapping is one-to-one. 
For $(x,y) \in [1,\infty)^2$, let us write $x = \cosh \t$ and $y = \cosh \phi$, $0 \le \t, \phi \le \pi$. Then 
it is easy to verify that
\begin{equation} \label{theta-phi}
    2 x y = \cosh (\t - \phi) + \cosh (\t + \phi), \quad x^2+y^2 -1 = \cosh (\t -  \phi) \cosh (\t + \phi), 
\end{equation}
from which it follows readily that $(2xy, x^2+y^2 -1) \in \Omega$ whenever $(x,y) \in \triangle$. 
The Jacobian of the change of variables $u= 2xy$ and $v = x^2+y^2 -1$ is $4 |x^2-y^2|$, so that
the mapping is a bijection. Since $dudv = 4 |x^2-y^2 |dxdy$ and the area of $G$ is four 
times of $\triangle$, the formula \eqref{theta-phi} follows from the change of variables,  the 
integral \eqref{Para-Square} and the fact that $ f(2xy, x^2+y^2 -1)$ is central symmetric
on $G$.
\end{proof}

In the case of $W_\g = W_{\a,\g}$, we denote the weight function $\CW_\g$ by $\CW_{\a,\g}$, which
is given explicitly by 
\begin{equation} \label{CW_a,g}
  \CW_{\a,\g} (x,y) := 4^\g |x-y|^{2 \a} (x^2-1)^\g (y^2-1)^\g |x^2-y^2| e^{-2xy-2}. 
\end{equation}

For the weight function $W_\g$ on the unbounded domain $G$, minimal cubature rules of degree $4n-1$ 
exist, as shown in the next theorem. To state the theorem, we need a notation. For $n \in \NN_0$, let $x_{k,n}$ 
be the zeros of the orthogonal polynomial $p_n(w;x)$, which are in the support set $[1,\infty)$ of the weight 
function $w$. We define $\theta_{k,n}$ by 
$$
       x_{k,n} = \cosh \t_{k,n}, \qquad 1 \le k \le n. 
$$
Since $x_{k,n} > 1$, it is evident that $\t_{k,n}$ are well defined. We then define 
\begin{align} \label{stjk}
  s_{j,k}: =  \cosh \tfrac{\t_{j,n}-\t_{k,n}}{2} \quad\hbox{and}\quad
                t_{j,k} := \cosh \tfrac{\t_{j,n} + \t_{k,n}}{2}.
\end{align}

\begin{thm} \label{thm:cubaCW}
For $\CW_{-\frac12}$ on $G$, we have the minimal cubature rule of degree $4n-1$ 
with $\dim \Pi_{2n-1}^2 + n$ nodes, 
\begin{align} \label{MinimalCuba2-}
  c_w^2 \int_{G}  f(x, y) \CW_{-\frac12}(x,y) dx dy = & \frac14 \sum_{k=1}^n \mathop{ {\sum}' }_{j=1}^k 
     \l_{j,n} \l_{k,n} \left[ f( s_{j,k}, t_{j,k})+  f( t_{j,k}, s_{j,k})    \right .  \\  
        & \quad + \left.  f( - s_{j,k}, - t_{j,k})+  f(- t_{j,k},- s_{j,k})  \right].      \notag
\end{align}
For $\CW_{\frac12}$ on $G$, we have the minimal cubature rule of degree $4n-3$ with 
$\dim \Pi_{2n-3}^2 + n$ nodes, 
\begin{align} \label{MinimalCuba2+}
  c_w^2 \int_{G}  f(x, y) \CW_{\frac12}(x,y) dx dy = & \frac14 \sum_{k=2}^n \sum_{j=1}^{k-1} 
     \l_{j,k}  \left[ f( s_{j,k}, t_{j,k})+  f( t_{j,k}, s_{j,k})    \right .  \\  
        & \quad + \left.  f( - s_{j,k}, - t_{j,k})+  f(- t_{j,k},- s_{j,k})  \right], \notag
\end{align} 
where $\lambda_{j,k} = \l_{j,n} \l_{k,n} (\cosh \t_{j,n} - \cosh \t_{k,n})^2$. 
\end{thm}

\begin{proof}
We prove only the case of $\CW_{-\f12}$, the case of $\CW_{\f12}$ is similar. Our starting point 
is the Gaussian cubature rule in \eqref{GaussCuba-}, which gives, by \eqref{Int-P-Q},
\begin{equation*}
   c_w^2 \int_{G} f(2xy, x^2+y^2-1) \CW_{-\frac12}(x,y) dx dy = \sum_{k=1}^n \mathop{ {\sum}' }_{j=1}^k  
          \l_{j,n} \l_{k,n} f(u_{j,k}, v_{j,k}), \quad f \in \Pi_{2n-1}^2. 
\end{equation*}
By \eqref{theta-phi} or by direct verification,
\begin{align*}
    \cosh \t_{j,n} + \cosh \t_{k,n} & = 2 \cosh \tfrac{\t_{j,n}-\t_{k,n}}{2}
        \cosh \tfrac{\t_{j,n} + \t_{k,n}}{2}  \\
   \cosh \t_{j,n} \cosh \t_{k,n} & = \cosh^2 \tfrac{\t_{j,n}-\t_{k,n}}{2} +
        \cosh^2 \tfrac{\t_{j,n} + \t_{k,n}}{2} -1 
\end{align*}
which implies that 
$$
 u_{j,k} = x_{j,n} + x_{k,n} = 2 s_{j,k}t_{j,k} \quad\hbox{and}\quad v_{j,k}  = x_{j,n} x_{k,n} = s_{j,k}^2+t_{j,k}^2 -1.
$$
Consequently, the above cubature rule can be written as 
\begin{equation*}
   c_w^2 \int_{G} f(2xy, x^2+y^2-1) \CW_{-\frac12}(x,y) dx dy = \sum_{k=1}^n \mathop{ {\sum}' }_{j=1}^k  
          \l_k \l_j f(2 s_{j,k}t_{j,k}, s_{j,k}^2 + t_{j,k}^2 -1) 
\end{equation*}
for all $f \in \Pi_{2n-1}^2$. For $f \in \Pi_{2n-1}^2$, the polynomial $f(2xy x^2+y^2-1)$ is of degree $4n-1$. 
Since the polynomials $f(2xy, x^2+y^2-1)$ are symmetric polynomials and all symmetric polynomials in
$\Pi_{4n-1}^2$ can be written in this way, we have established \eqref{MinimalCuba2-} for symmetric
polynomials. By the Sobolev's theorem on invariant cubature rules, this establishes \eqref{MinimalCuba2-} 
for all polynomials in $\Pi_{4n-1}^2$. 
\end{proof}

The number of nodes of the cubature rule in \eqref{MinimalCuba2-} is precisely 
$$
    N = \dim \Pi_{2n-1}^2+n =  \dim \Pi_{2n-1}^2 +  \left \lfloor \frac{2n}2  \right \rfloor 
$$
which is the lower bounded of \eqref{lwbd} with $n$ replaced by $2n$. Thus, \eqref{MinimalCuba2-}
attains the lower bound \eqref{lwbd}. Similarly, \eqref{MinimalCuba2+} attains the lower bound 
\eqref{lwbd} with $n$ replaced by $2n-1$. 

It turns out that a basis of orthogonal polynomials with respect to $\CW_{\g}$ can be given explicitly. 
We need to define three more weight functions associated with $W_\g$.  
\begin{align}\label{Wij}
\begin{split}
  W_\g^{(1,1)}(u,v) &  := (1-u+v)(1+u+v)W_\g(u,v), \\
  W_\g^{(1,0)}(u,v)  & := (1-u+v)W_\g(u,v), \qquad \qquad (u,v) \in \Omega,\\
  W_\g^{(0,1)}(u,v)  & :=  (1+u+v)W_\g(u,v).
\end{split}  
\end{align} 
Under the change of variables $u = x+y$ and $v = xy$, $1-u+v = (x-1)(y-1)$ and $1+u+v = (x+1)(y+1)$. 
The three weight functions in \eqref{Wij} are evidently of the same type as $W_\g$. We denote by 
$\{P_{k,n}^{(\g)}: 0 \le k \le n\}$ an orthonormal basis of $\CV_n(W_\g)$ under $ \la \cdot, \cdot \ra_{W_\g}$. 
For $0 \le k \le n$, we further denote by $P_{k,n}^{(\g),1,1}$, $P_{k,n}^{(\g),1,0,}$, $P_{k,n}^{(\g),0,1}$ the 
orthonormal polynomials of degree $n$ with respect to $\la f, g \ra_W$ for $W = W_\g^{(1,1)}$, 
$W_\g^{(1,0)}$, $W_\g^{(0,1)}$, respectively. 
  
\begin{thm} \label{thm:Qop}
For $n = 0,1,\ldots$, a mutually orthogonal basis of $\CV_{2n}(\CW_{\g})$ is given by
\begin{align} \label{Qeven2}
\begin{split}
         {}_1Q_{k,2n}^{(\g)}(x,y):= & P_{k,n}^{(\g)}(2xy, x^2+y^2 -1), \qquad 0 \le k \le n, \\
         {}_2Q_{k,2n}^{(\g)}(x,y) := &  (x^2-y^2)  P_{k,n-1}^{(\g),1,1}(2xy, x^2+y^2 -1),  \quad 0 \le k \le n-1, 
 \end{split}
\end{align}
and a mutually orthogonal basis of $\CV_{2n+1}(\CW_{\g})$ is given by 
\begin{align} \label{Qodd2}
 \begin{split}
     &  {}_1Q_{k,2n+1}^{(\g)}(x,y):=  (x+ y) P_{k,n}^{(\g),0,1}(2xy, x^2+y^2 -1), \qquad  0 \le k \le n,  \\
     &  {}_2Q_{k,2n+1}^{(\g)}(x,y):=  (x-y) P_{k,n}^{(\g),1,0}(2xy, x^2+y^2 -1), \qquad 0 \le k \le n. 
 \end{split}
\end{align}
\end{thm}

Under the mapping $(u,v) \mapsto (2xy, x^2+y^2-1)$, it is easy to see that $W_\g^{(1,1)}(u,v)$
becomes $(x^2-y^2)^2 \CW_\g(x,y)$, $W_\g^{(1,0)}(u,v)$ becomes $(x-y)^2 \CW_\g(x,y)$,
and  $W_\g^{(0,1)}(u,v)$ becomes $(x+y)^2 \CW_\g(x,y)$. Hence, using Lemma \ref{Int-Para-cube}, 
the proof can be deduced from the orthogonality of $P{k,n}^{(\g), i,j}$ and symmetry of the integrals 
against $W_\g^{(i,j)}$, similar to the proof of Theorem 3.4 in \cite{X12b}.

Combining with \eqref{OP-1/2} and \eqref{OP+1/2}, we can express orthogonal polynomials
${}_iQ_{k,n}^{(\pm \f12)}$ in terms of orthogonal polynomials in one variables. For example,
in the case of $\CW_{\a,\g}$ in \eqref{CW_a,g}, we can express even degree orthogonal 
polynomials in terms of the Laguerre polynomials. 
 
\begin{prop}
Let $\a > -1$. A mutually orthogonal basis of $\CV_{2n}(\CW_{\a,-\frac12})$ is given 
by, for $0 \le k \le n$ and $0 \le k \le n-1$, respectively,  
\begin{align*} 
    & {}_1Q_{k,2n}^{(\a,-\frac12)}(\cosh\t+1,\cosh \phi +1) = L_n^{(\a)} (\cosh (\t - \phi)) L_k^{(\a)}(\cosh (\t+\phi))  \\
    &    \qquad\qquad  \qquad\qquad \qquad\qquad  \qquad 
          +  L_k^{(\a)} (\cosh (\t - \phi)) L_n^{(\a)}(\cosh (\t+\phi)), \\
    & {}_2Q_{k,2n}^{(\a,-\frac12)}(\cosh\t +1,\cosh \phi+1)  = 
      (x^2-y^2) \left[ L_{n-1}^{(\a+1)} (\cosh (\t - \phi)) L_k^{(\a+1)} (\cosh (\t+\phi)) \right. \\ 
    &    \qquad\qquad  \qquad\qquad \qquad\qquad  \qquad 
        \left. +  L_k^{(\a+1)} (\cosh (\t - \phi)) L_{n-1}^{(\a+1)} (\cosh (\t+\phi)) \right].
\end{align*}
\end{prop}

In these formulas we used $x = \cosh \t -1$ and $y = \cosh \phi -1$, so that $L_n^\a(x-1) = L_n^\a (\cosh \t)$
and $L_n^\a(y-1) = L_n^\a (\cosh \phi)$ and we can then use \eqref{theta-phi}. Note, however, we cannot 
express the odd degree ones in terms of Laguerre polynomials. In fact, for ${}_2Q_{k, 2n+1}$, we need 
orthogonal polynomials of one variable with respect to $w(x) = (x+1)(x-1)^\a e^{- (x-1)}$, which is not a 
shift of the Laguerre weight function. 

\begin{cor}
For the weight function $\CW_{-\f12}$, the nodes of the minimal cubature formulas \eqref{MinimalCuba2-} 
are common zeros of orthogonal polynomials $\{ {}_1Q_{k,2n}^{(-\frac12)}: 0 \le k \le n\}$. 
\end{cor}

An analogue result can be stated for the minimal cubature formulas \eqref{MinimalCuba2+}.

\end{document}